\documentclass[12pt]{amsart}

\usepackage{amssymb}
\usepackage{bm}
\usepackage{graphicx}
\usepackage{psfrag}
\usepackage{color}
\usepackage{url}
\usepackage{enumerate}

\newcommand{\excise}[1]{}

\DeclareMathOperator{\Hom}{Hom}

\newtheorem{thm}{Theorem}[section]
\newtheorem{lemma}[thm]{Lemma}

\newtheorem{cor}[thm]{Corollary}
\newtheorem{prop}[thm]{Proposition}
\newtheorem{conj}[thm]{Conjecture}

\newtheorem{question}[thm]{Question}
\newtheorem{prob}[thm]{Problem}

\theoremstyle{definition}
\newtheorem{example}[thm]{Example}
\newtheorem{remark}[thm]{Remark}
\newtheorem{defn}[thm]{Definition}

\numberwithin{equation}{section}

\newcommand{\ring}[1]{\ensuremath{\mathbb{#1}}}

\renewcommand\>{\rangle}

\newcommand\NN{\ring{N}}

\newcommand\QQ{\ring{Q}}

\newcommand\ZZ{\ring{Z}}

\newcommand\iso{\cong}

\begin{document}

\mbox{}
\title{Factorization Properties of Leamer Monoids\qquad}

\author[J. Haarmann]{Jason Haarmann}
\address{Eastern Illinois University \\ Charleston, IL 61920}
\email{jhaarmann@eiu.edu}

\author[A. Kalauli]{Ashlee Kalauli}
\address{University of Hawai`i at Hilo \\ Hilo, HI 96720}
\email{h8askala@gmail.com}

\author[A. Moran]{Aleesha Moran}
\address{McKendree University \\ Lebanon, IL 62254}
\email{aleeshavmoran@gmail.com}

\author[C. O'Neill]{Christopher O'Neill}
\address{Mathematics Department\\Duke University\\Durham, NC 27708}
\email{musicman@math.duke.edu}

\author[R. Pelayo]{Roberto Pelayo}
\address{Mathematics Department\\University of Hawai`i at Hilo\\Hilo, HI 96720}
\email{robertop@hawaii.edu}

\date{}

\begin{abstract}
\hspace{-2.05032pt}
The Huneke-Wiegand conjecture has prompted much recent research in Commutative Algebra.  In studying this conjecture for certain classes of rings, Garc\'ia-S\'anchez and Leamer construct a monoid $S_\Gamma^s$ whose elements correspond to arithmetic sequences in a numerical monoid $\Gamma$ of step size $s$.  These monoids, which we call Leamer monoids, possess a very interesting factorization theory that is significantly different from the numerical monoids from which they are derived.  In this paper, we offer much of the foundational theory of Leamer monoids, including an analysis of their atomic structure, and investigate certain factorization invariants.  Furthermore, when $S_\Gamma^s$ is an arithmetical Leamer monoid, we give an exact description of its atoms and use this to provide explicit formulae for its Delta set and catenary degree.
\end{abstract}
\maketitle

\section{Introduction}\label{s:intro}

In~\cite{tensor}, C.~Huneke and R. Wiegand propose the following conjecture regarding torsion submodules of tensor products.

\begin{conj}[Huneke-Wiegand]\label{c:HW}
Let $R$ be a one-dimensional Gorenstein domain.  Let $M \neq 0$ be a finitely generated $R$-module, which is not projective.  Then the torsion submodule of $M~\!\!\otimes_R~\!\!\Hom_R(M,R)$ is non-trivial.  
\end{conj}

Recently, this still-open conjecture has spurred much subsequent work (see~\cite{auslander, vanishing, intersection, goto, leamer}).  Of particular interest is~\cite{intersection}, where P. Garc\'ia-S\'anchez and M. Leamer study this conjecture in special cases related to numerical monoid algebras.  Given a numerical monoid $\Gamma$ and $s \in \NN \setminus \Gamma$, they construct a monoid, which we denote as $S_\Gamma^s$, whose elements correspond to arithmetic sequences in $\Gamma$ of step size $s$ and whose monoid operation is set-wise addition.  These monoids, which we refer to as \emph{Leamer monoids} in honor of \cite{leamer}, reduce a special case of the Huneke-Wiegand conjecture to finding irreducible elements of a certain type.

\begin{prop}[{\cite[Corollary 7]{intersection}}]\label{p:HW}
Let $\Gamma$ be a numerical monoid and $\mathbb K$ be a field.  The monoid algebra $\mathbb K[\Gamma]$ satisfies the Huneke-Wiegand conjecture for monomial ideals generated by two elements if and only if for each $s \in \NN \setminus \Gamma$, there exists an irreducible arithmetic sequence of the form $\{x, x+s, x+2s\}$ in $\Gamma$.  
\end{prop}

Thus, understanding the atomic structure of Leamer monoids could provide progress towards proving the Huneke-Wiegand conjecture.  In investigating these algebraic objects, we find that, beyond their importance in Commutative Algebra, Leamer monoids possess a very interesting factorization theory with numerous notable properties.   We investigate several factorization invariants of $S_\Gamma^s$, including a computation of elasticity and a bound on the Delta set. In special cases, we provide exact formulae for length sets, Delta sets, and catenary degrees.  

Although elements in Leamer monoids are arithmetic sequences in $\Gamma$ of a fixed step size, the atomic properties of $S_\Gamma^s$ differ greatly from the numerical monoid $\Gamma$ from which they are derived. In particular, many of the computable invariants for numerical monoids (e.g.~Delta sets, see~\cite{delta}) are harder to establish in Leamer monoids.  Thus, investigating special cases (i.e.~arithmetical Leamer monoids) becomes both necessary and fruitful.

As Leamer monoids are a novel construction, this paper provides much of the foundational theory for studying these intriguing algebraic objects.  In Section~\ref{s:defn}, we provide several definitions and examples of Leamer monoids and investigate their atomic structure, showing that although they possess infinitely many irreducible elements, these atoms are still fairly constrained.  In Section~\ref{s:elasticitydelta}, we show that Leamer monoids have infinite elasticity and give an explicit bound for the maximal element of their Delta sets.  In Section~\ref{s:arithmetical}, we investigate the special case of arithmetical Leamer monoids, that is, Leamer monoids $S_\Gamma^s$ where $\Gamma$ is generated by an arithmetic sequence with step size~$s$.  For these Leamer monoids, we provide an exact description of the irreducible and reducible elements and use this to derive formulae for their Delta sets and catenary degrees.  In Section~\ref{s:future}, we provide several open problems related to factorization invariants of Leamer monoids and the study of $S_\Gamma^s$ when $\Gamma$ has certain types of generating sets.

\section{Definitions and Atomic Structure of Leamer Monoids}\label{s:defn}%

In what follows, we let $\Gamma = \langle m_1, \ldots, m_k\rangle$ be a numerical monoid (i.e., co-finite additive submonoid of $\mathbb N$) where $\{m_1, \ldots, m_k\}$ constitutes its minimal generating set and $m_1 < m_2 < \ldots < m_k$.

\begin{defn}\label{d:leamer}
Given a numerical monoid $\Gamma$, $s \in \mathbb N\setminus \Gamma$, define 
$$S_\Gamma^s = \{(0,0)\} \cup \left\{(x,n) :  \{x,x+s, x+2s, \ldots, x+ns\} \subset \Gamma \right\} \subset \mathbb N^2.$$
That is, $S_\Gamma^s$ is the collection of arithmetic sequences of step size $s$ contained in $\Gamma$.  
\end{defn}

Since $S_\Gamma^s$ is a subset of $\NN^2$, it inherits its natural component-wise additive structure.  It is clear that $S_\Gamma^s$ is closed under this operation, so $S_\Gamma^s$ is an additive submonoid of $\NN^2$.  Throughout this paper, we denote by $\mathcal A (S_\Gamma^s)$ the set of irreducible elements (or atoms) of $S_\Gamma^s$.  Building upon functions in the \texttt{GAP} package \texttt{numericalsgps} \cite{numericalsgps}, we are able to plot the elements of $S_\Gamma^s$ and indicate (by large red dots) the irreducible elements; see Figures~\ref{f:leamergraph}~and~\ref{f:anotherleamergraph}.  

\begin{example}\label{e:7_10_s=3}
For $\Gamma = \langle 7,10 \rangle$ and $s = 3$, the graph of the Leamer monoid $S_{\langle 7,10\rangle}^3$ is given below in Figure~\ref{f:leamergraph}.  In this example, $\Gamma$ is generated by an arithmetic sequence with step size $3$, which is equal to $s$.  The Frobenius number of $\Gamma$ (53 in this case) can be easily read off the graph as the rightmost column absent of any dots.
\end{example}

\begin{figure}[ht]
  \centering
    \includegraphics[width=4in]{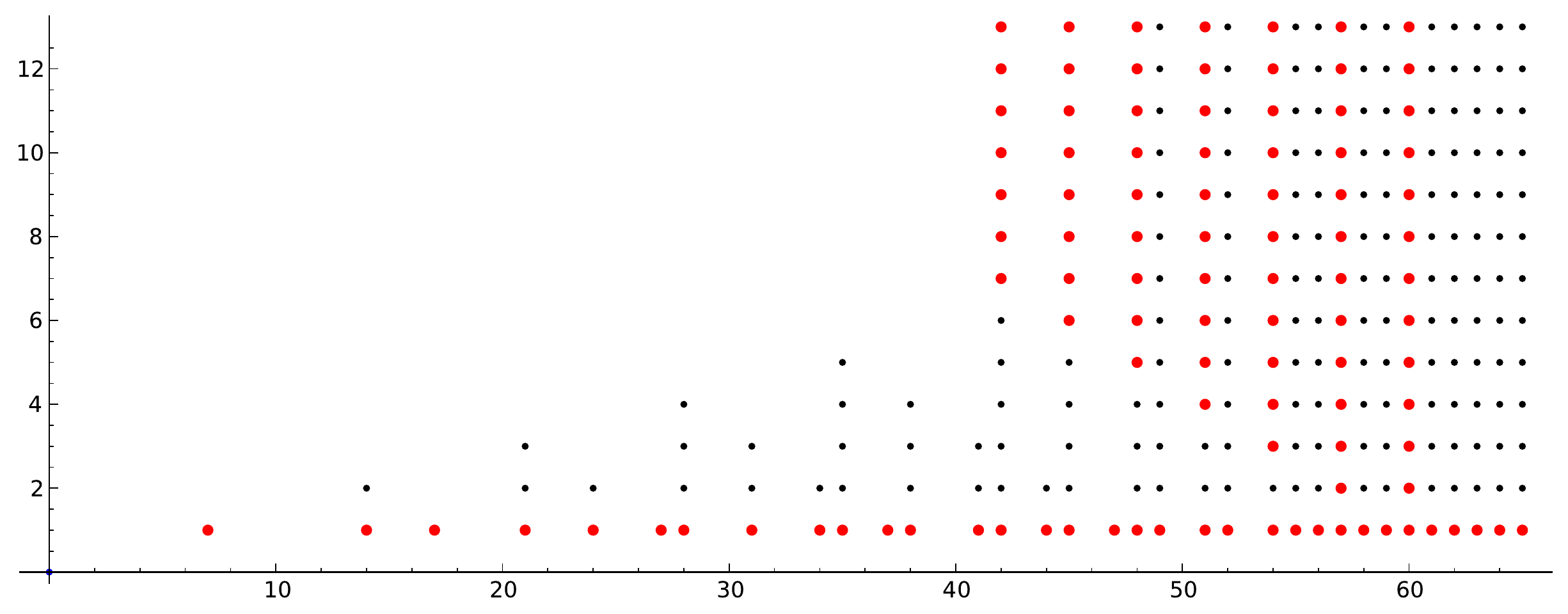}
  \caption{The Leamer Monoid described in Example~\ref{e:7_10_s=3}.}
\label{f:leamergraph}
\end{figure}

As in Example~\ref{e:7_10_s=3}, Leamer monoids $S_\Gamma^s$ where $\Gamma$ is of the form $\Gamma = \<m, m + s, \ldots, m + ks\>$ for some $m, k \in \NN$ have the best understood set of irreducibles (see Section~\ref{s:arithmetical} for a detailed discussion).  When $\Gamma$ is not generated by an arithmetic sequence (or even when it is, but $s$ is not equal to the step size between the generators), the Leamer monoids have less predictable structure.

\begin{example}\label{e:13_17_22_40_s=4}
For $\Gamma = \langle 13,17,22,40 \rangle$ and $s = 4$, the Leamer monoid $S_{\Gamma}^4$ is given in Figure~\ref{f:anotherleamergraph}.  Unlike Example~\ref{e:7_10_s=3}, this monoid possesses numerous irreducibles with $n \geq 2$ before the left-most infinite column.  These atoms make the analysis of such Leamer monoids much more difficult.  
\end{example}

\begin{figure}[ht]
  \centering
    \includegraphics[width=4in]{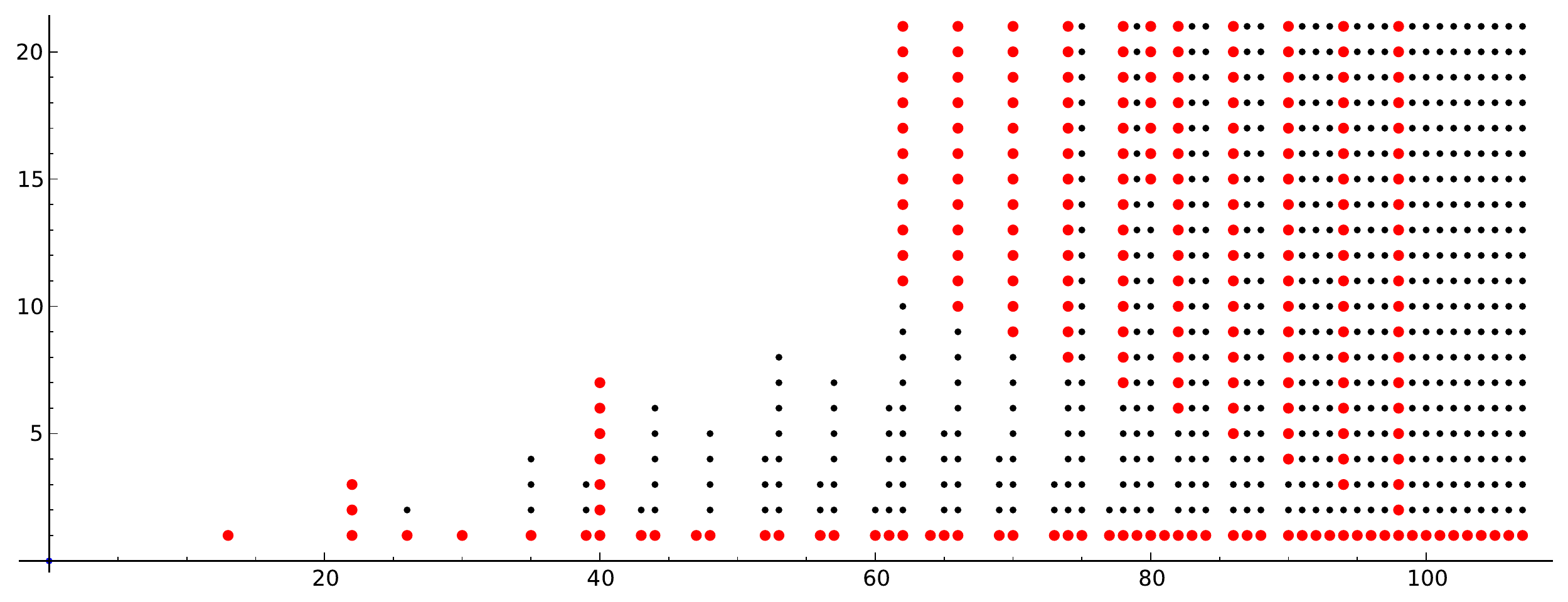}
  \caption{The Leamer Monoids described in Example~\ref{e:13_17_22_40_s=4}.}
\label{f:anotherleamergraph}
\end{figure}

Motivated by these graphical representations, we present the following definitions that provide insight into the structure of Leamer monoids.

\begin{defn}\label{d:column}
For a Leamer monoid $S_{\Gamma}^s$ and $x \in \Gamma$, the \emph{column at $x$} is the set 
$$\{(x,n) \in S_{\Gamma}^s:n \geq 1\} .$$
If this set is empty, we say that no column exists at $x$.  If a column exists at $x$, the column at $x$ is said to be \emph{finite} (resp., infinite) if the column has finite (resp., infinite) cardinality.  The \emph{height} of the finite column at $x$ is 
$$\max\{n:(x,n) \in S_{\Gamma}^s\}.$$
\end{defn}

\begin{remark}
Definition~\ref{d:leamer} makes sense for any $s \in \ZZ$.  However, if $s \in \Gamma$, then the atomic structure of $S_\Gamma^s$ is inherited directly from $\Gamma$.  
Furthermore, $S_\Gamma^s \iso S_\Gamma^{-s}$ as monoids, and the latter corresponds to the (equivalent) embedding of $S_\Gamma^s$ in $\NN^2$ where the point $(x,n)$ represents the arithmetic sequence $\{x - ns, \ldots, x - s, x\}$ rather than $\{x, x + s, \ldots, x + ns\}$.  Many of the statements in Lemma~\ref{l:atoms} would be different under this new embedding; for instance, every column of $S_\Gamma^{-s}$ is finite.  
\end{remark}

The first non-identity element in a Leamer monoid and the first infinite column play important computational roles in understanding the monoid structure of $S_{\Gamma}^s$.  We provide their notation below.

\begin{defn}\label{d:x0}
Given a Leamer monoid $S_{\Gamma}^s$, we use $x_0(S_\Gamma^s)$ to denote the smallest $x$ such that $(x,1) \in S_{\Gamma}^s$.  We denote by $x_f(S_\Gamma^s)$ the first infinite column of $S_\Gamma^s$, that is, the smallest $x$ such that $(x,n) \in S_{\Gamma}^s$ for all $n \geq 1$.  When $S_\Gamma^s$ is clear from the context, we refer to these as simply $x_0$ and $x_f$.  
\end{defn}

If $\mathcal F (\Gamma)$ is the Frobenius number of $\Gamma$, then clearly there exists an infinite column at $\mathcal F(\Gamma) + 1$.  In particular, both $x_0$ and $x_f$ exist; see Lemma~\ref{l:atoms} below.  

Since the connection between Leamer monoids and the Huneke-Wiegand conjecture centers on the existence of irreducible elements of length $2$, understanding the factorization structure of $S_{\Gamma}^s$ is of significant importance.   The first important observation is that Leamer monoids are atomic.  We leave the proof of this to the reader.  

\begin{thm}\label{t:atomic}
Let $\Gamma$ be a numerical monoid with $s \in \mathbb N \setminus \Gamma$. The Leamer monoid $S_\Gamma^s$ is atomic.  That is, any $(x,n) \in S_{\Gamma}^s$ can be written as a finite sum of irreducible elements.
\end{thm}

The following observations are motivated by the graphs in Figures~\ref{f:leamergraph} and~\ref{f:anotherleamergraph}. 
\begin{lemma}\label{l:atoms}
Let $S_{\Gamma}^s$ be a Leamer monoid.  

\begin{enumerate}[(a)]

\item 
If $(x,1) \in S_{\Gamma}^s$, then $(x,1) \in \mathcal A(S_\Gamma^s).$

\item 
For $n >\!\!> 0$, $(x_f,n) \in \mathcal A(S_\Gamma^s)$.

\item 
If $(x,n) \in S_\Gamma^s$, then $(x,n') \in S_\Gamma^s$ for all $1 \leq n' \leq n$.

\item 
If $(x,n-1) \in \mathcal A(S_\Gamma^s)$ and $(x,n) \in S_\Gamma^s$ for $n > 2$, then $(x,n) \in \mathcal A(S_\Gamma^s).$ 

\item 
If $(x,n-1) \in \mathcal A(S_\Gamma^s)$ and $(x-s,n) \in S_\Gamma^s$ for $n > 2$, then $(x-s,n) \in \mathcal A(S_\Gamma^s).$ 

\item 
The column at every $x > \mathcal F(\Gamma)$ is infinite.

\item 
For all $x > \mathcal F(\Gamma) + x_0$ and $n \geq 2$, $(x,n)$ is a reducible element in $S_\Gamma^s$. 

\end{enumerate}
\end{lemma}

\begin{proof}
If $(x,n) = (y,p) + (z,q)$ with $n > 2$, then at least one of $p$ or $q$ is greater than $1$, yielding a factorization for $(x,n-1)$.  This proves the contrapositive of (d).  For (g), notice that for $n \geq 2$ and $x > \mathcal F(\Gamma) + x_0$, we can write $(x,n) = (x_0,1) + (x-x_0, n-1)$.  The remaining proofs are left to the reader.  
\end{proof}

The above lemma motivates the following definitions.

\begin{defn}\label{d:atoms}
Fix a Leamer monoid $S_\Gamma^s$ and an  $x \in \NN$ whose column in $S_\Gamma^s$ is non-empty.  The element $(x,1) \in S_\Gamma^s$ is called \emph{trivial}.  The column at $x$ is called \emph{purely irreducible} (or \emph{purely atomic}) if it contains only atoms, \emph{mixed} if it contains both reducible and non-trivial irreducible elements, and \emph{purely reducible} otherwise.
\end{defn}

\begin{example}
In Figure~\ref{f:leamergraph}, the red dots with height 1 are all trivial atoms, whereas the element $(42,7)$ is a nontrivial atom.  The column at 60 is purely atomic, since every element is irreducible, and the column at 52 is purely reducible, since its only atom has height 1.  The column at 48 has 3 reducible elements, but every element with height at least 5 is irreducible, so this column is mixed.  
\end{example}

\begin{remark} \label{r:bound}
The Leamer monoid in Example~\ref{e:7_10_s=3} has $x_0 = 7$ and $\mathcal F (\Gamma) = 53$ and has non-trivial irreducibles at $x = \mathcal F(\Gamma) + x_0 = 60$.  Thus, the bound presented in Lemma~\ref{l:atoms}(g) is sharp.  On the other hand, the Leamer monoid in Example~\ref{e:13_17_22_40_s=4} has $\mathcal F(\Gamma) + x_0 = 89 +13 = 104$ with the last non-trivial irreducible occurring at $98$.  Thus, the bound is not always reached.  
\end{remark}

\section{Elasticities and Delta Sets}\label{s:elasticitydelta}

In atomic, cancellative, commutative monoids, various invariants have been used to measure how far an element is from have unique factorization into irreducible elements.  In this section, we investigate several of these invariants for Leamer monoids.  

We begin by analyzing the elasticity of Leamer monoids.  For more detail on elasticity, see~\cite[Definition~1.4.1]{bible}.  

\begin{defn}\label{d:elasticity}
Fix a Leamer monoid $S_\Gamma^s$.  The \emph{set of lengths} of $(x,n) \in S_\Gamma^s$ is given by 
$$\mathcal L(x,n) = \{r : (x,n) = \textstyle\sum_{i = 1}^r (x_i,n_i), (x_i,n_i) \in \mathcal A(S_\Gamma^s)\}.$$
Let $\ell(x,n) = \min \mathcal L(x,n)$ and $L(x,n) = \max\mathcal L(x,n)$ denote the minimum and maximum factorization lengths of $(x,n)$, respectively.  The \emph{elasticity} of $(x,n)$ is given by $\rho(x,n) = L(x,n)/\ell(x,n)$, and the \emph{elasticity} of $S_\Gamma^s$ is given by 
$$\rho(S_\Gamma^s) = \sup\{\rho(x,n) : (x,n) \in S_\Gamma^s\}.$$
\end{defn}

We begin by showing that every Leamer monoid has infinite elasticity.

\begin{thm}\label{t:infelasticity}
For any Leamer monoid $S_\Gamma^s$, $\rho(S_\Gamma^s) = \infty.$
\end{thm}

\begin{proof}
For $t >\!\!> 0$, we have 
$$(t \cdot x_f, t) = t(x_f,1) = ((t-1) \cdot x_f, 1) + (x_f,t-1)$$
so $\rho(t \cdot x_f, t) = \frac{t}{2}$.  Letting $t$ tend to infinity completes the argument.  
\end{proof}

Variants of elasticity also appear in factorization theory literature (see~\cite{fullelas} and~\cite{katoms}).  For convenience, we define these here.  

\begin{defn}\label{d:fullyelastic}
The \emph{$k$-th refined elasticity of $S_\Gamma^s$} is defined by 
$$\rho_k(S_\Gamma^s) = \sup\left\{\rho(x,n) \, : \, (x,n) \in S_\Gamma^s \text{ nonzero with } k \in \mathcal L(x,n)\right\}$$  
for $k \ge 2$.  We say $S_\Gamma^s$ is \emph{fully elastic} if $\{\rho(x,n) : (x,n) \in S_\Gamma^s\} = \QQ \cap [1,\infty)$.  
\end{defn}

We now compute the refined elasticity for Leamer monoids.  

\begin{thm}\label{t:fullyelastic}
Let $S_\Gamma^s$ be a Leamer monoid.  For all $k \geq 2$, $\rho_k(S_\Gamma^s) = \infty$.   Moreover, $S_\Gamma^s$ is never fully elastic.
\end{thm}

\begin{proof}
Fix $n_f$ such that $(x_f,n_f)$ is an atom in $S_\Gamma^s$.  For $t \ge k \ge 2$, we have 
$$t \cdot (x_f, n_f) = ((k-1)x_f,tn_f - 1) + ((t - k + 1)x_f, 1)$$
so $\rho(tx_f, tn_f) = \frac{t}{2}$ and $k \in \mathcal L(tx_f, tn_f)$.  In fact, whenever $n > n_f$ and $x > x_f + \mathcal F(\Gamma)$, we have $(x,n) = (x_f,n - 1) + (x - x_f, 1)$, so $\ell(x,n) = 2$.  This shows 
$$\max\{\ell(x,n) : (x,n) \in S_\Gamma^s\} < \infty$$
from which the second claim follows.  
\end{proof}

We continue this section with a discussion on the Delta sets of Leamer monoids.  In particular, we show that the Delta set of any Leamer monoid is finite, and we give a method to compute its maximal element.  For more information on Delta sets of commutative cancellative monoids, see \cite[Section~1.4]{bible}.

\begin{defn}
Fix a Leamer monoid $S_\Gamma^s$.  The \emph{Delta set} of $(x,n) \in S_\Gamma^s$ is given by 
$$\Delta(x,n) = \{\ell_i - \ell_{i-1} : 2 \le i \le k\},$$
where $\mathcal L(x,n) = \{\ell_1, \ldots, \ell_k\}$.  The \emph{Delta set} of $S_\Gamma^s$ is given by 
$$\Delta(S_\Gamma^s) = \textstyle\bigcup_{(x,n) \in S_\Gamma^s} \Delta(x,n).$$
\end{defn}

\begin{thm}\label{t:nstar}
Fix a Leamer monoid $S_\Gamma^s$.  Let 
$$C = \{(x,n) \in S_\Gamma^s \setminus \mathcal A(S_\Gamma^s) : (x,n+1) \notin S_\Gamma^s \setminus \mathcal A(S_\Gamma^s)\}$$
that is, the set of reducible elements which lie in finite or mixed columns and have maximal height.  Then $|\Delta(S_\Gamma^s)| <\infty$, and in fact $\max\Delta(S_\Gamma^s) \leq n^* - 1$, where 
$$n^* = \max\{n : (x,n) \in C\}.$$
\end{thm}

\begin{proof}
Fix a reducible element $(x,n) \in S_\Gamma^s$.  By~\cite[Lemma 4.1]{schaeffer}, it suffices to show that $\ell(x,n) \le n^* + 1$.  If $n \le n^* + 1$, then so is $\ell(x,n)$.  If $n \ge n^* + 2$, the column at $x$ is infinite, so we can write $x = y + z$ where there is a column at $y$ and an infinite column containing nontrivial irreducibles at $z$.  Since $(x,n) = (y,1) + (z,n-1)$ is a sum of atoms, we have $\ell(x,n) = 2$.  
\end{proof}

We now examine the length set $\mathcal L(x,n)$ for elements $(x,n) \in S_\Gamma^s$ with $x >\!\!> n$.  

\begin{prop}\label{p:colstab}
Fix a Leamer monoid $S_\Gamma^s$.  Let $s(x,n) = \frac{n}{x}$ for $(x,n) \in \NN^2$, and let 
$$s_L = \max\{s(x,n) : (x,n) \in S_\Gamma^s, (x,n+1) \notin S_\Gamma^s\}.$$
For each $(x,n) \in S_\Gamma^s \setminus \mathcal A(S_\Gamma^s)$ with $n > \lfloor s_L x \rfloor$, we have $\mathcal L(x,n) = \mathcal L(x,n+1) = \{2, 3, \ldots, h\}$ for some integer $h \ge 2$.  In particular, $\Delta(x,n) \subseteq \{1\}$.  
\end{prop}

\begin{proof}
Fix $(x,n) \in S_\Gamma^s$ with $n > \lfloor s_L x \rfloor$, and a factorization $(x,n) = \sum_{i = 1}^k (x_i, n_i)$.  Then
$$x \cdot s_L < n = x \cdot s(x,n) = x \cdot s(x_1 + \cdots + x_k, n_1 + \cdots + n_k) \le x \cdot \textstyle\max_i\{s(x_i,n_i)\},$$
which gives $s_L < \max_i\{s(x_i,n_i)\}$, so some $(x_j, n_j)$ must reside in an infinite column with $n_j > 1$.  By Lemma~\ref{l:atoms}(d), we have $(x_j,a) \in \mathcal A(S_\Gamma^s)$ for all $a \ge n_j$.  This yields a factorization $(x,n) = (x_j,n-k+1) + \sum_{i \ne j} (x_i,1)$.  This means $(x,n+1) = (x_j,n-k+2) + \sum_{i \ne j} (x_i,1)$, and for $k \ge 3$, $(x,n) = (x_j,n-k+2) + (x_m + x_n,1) + \sum_{i \ne j,m,n} (x_i,1)$ for distinct indices $j, m, n  \le k$.  This proves both claims.  
\end{proof}

The proof of the following proposition is similar to that of Proposition \ref{p:colstab} and is left to the reader.  

\begin{prop}\label{p:rowstab}
Fix a Leamer monoid $S_\Gamma^s$.  For each $(x,n) \in S_\Gamma^s$ with $x > n\mathcal F(\Gamma)$, we have $\mathcal L(x,n) = \mathcal L(x+1,n)$.  In particular, $\Delta(x,n) = \Delta(x+1,n)$.  
\end{prop}


To conclude this section, we give a method to find $\max\Delta(S_\Gamma^s)$ by giving a bounded region in $\NN^2$ in which it must occur.  This effectively gives an algorithm to compute the maximum value in the Delta set for any given Leamer monoid; see Remark~\ref{r:deltaregion}.  First, we give a technical lemma.  

\begin{lemma}\label{l:deltafactlen}
Fix an atomic, cancellative monoid $M$ with $\lambda = 1 + \sup\{\ell(x) : x \in M\}$ finite.  For $k \in \Delta(M)$, there exists $z \in M$ with 
$$\mathcal L(z) \cap \{r, r+1, \ldots, r + k'\} = \{r, r+k'\}$$
for some $r \le \lambda - k$ and $k' \ge k$.  
\end{lemma}

\begin{proof}
Fix $w \in M$ with irreducible factorizations $w = a_1 \cdots a_\ell = b_1 \cdots b_{\ell+k}$ and no factorizations of length strictly between $\ell$ and $\ell+k$.  If $\ell + k > \lambda$, then $b_1 \cdots b_\lambda = c_1 \cdots c_r$ for some $r < \lambda$.  Let $z = c_1 \cdots c_r$.  This gives a factorization $w = c_1 \cdots c_r b_{\lambda + 1} \cdots b_{\ell + k}$, so we must have $r + \ell + k - \lambda \le \ell$.  In particular, $r \le \lambda - k$, so for $r$ maximal, $z$ cannot have a factorization with length strictly between $\lambda$ and $r$.  
\end{proof}



\begin{thm}\label{t:maxdelta}
Let $\lambda = 1 + \max\{\ell(x) : x \in S_\Gamma^s\}$, let $s_L$ be defined as in Proposition~\ref{p:colstab} and let $(x_i,n_i)$ denote a nontrivial irreducible in an infinite column.  Then 
$$\max\Delta(S_\Gamma^s) = \max\{\Delta(x,n) : x \le x_B, n \le \lfloor s_L x \rfloor\}$$
where $x_B = \mathcal F(\Gamma) + x_i + (n_i + \lambda)(\mathcal F (\Gamma) + x_0)$.  
\end{thm}

\begin{proof}
Fix $(x,n) \in S_\Gamma^s$.  If $n > \lfloor s_L x \rfloor$, then by Proposition~\ref{p:colstab}, we have $\Delta(x,n) \subseteq \{1\}$.  
Now suppose $x > x_B$.  If $n \le n_i + \lambda$, then by Proposition~\ref{p:rowstab}, $\Delta(x,n) = \Delta(x_B,n)$.  If $n > n_i + \lambda$, then for $r \le \lambda - 1$, we can write $(x,n) = (x_i, n - r - 1) + r \cdot (x_0,1) + (x - x_i - rx_0,1)$, so $\{2, \ldots, \lambda + 1\} \subset \mathcal L(x,n)$.  Thus, if $k \in \Delta(x,n)$, then by Lemma~\ref{l:deltafactlen} some delta set value $k' \geq k$ must occur in a column before $x$.  
\end{proof}


\begin{remark}\label{r:deltaregion}
While the Delta sets in the regions described in Propositions~\ref{p:colstab} and~\ref{p:rowstab} are very well behaved, nontrivial Delta set elements often occur throughout the remaining elements.  
This makes it very difficult to find a region on which the entire Delta set is obtained.   This problem is solved when the Delta set is an interval; see Question~\ref{q:deltainterval}.  
\end{remark}

\begin{example}\label{e:deltabound}
In many of the examples we computed, we found that $\max(\Delta(S_\Gamma^s)) = \lambda - 2$.  However, this equality does not always hold.  For instance, when $\Gamma = \langle 13,17,22,40 \rangle$ and $s = 4$, we have $\Delta(S_\Gamma^s) = \{1,2\}$ but $\lambda = 5$ (see Example~\ref{e:13_17_22_40_s=4}).  
\end{example}

\section{Leamer Monoids generated by arithmetic sequences}\label{s:arithmetical}

In this section, we discuss a Leamer monoid $S_\Gamma^s$ where $\Gamma$ is generated by an arithmetic sequence with step size $s$.  In particular, we give a complete characterization of the Leamer monoids of this form, and use this to give a closed form for their Delta sets.

\begin{defn}\label{d:arithmetical}
The Leamer monoid $S_\Gamma^s$ is \emph{arithmetical} if $\Gamma = \<m, m + s, \ldots, m + ks\>$ for some $m, k \in \NN$.  For $m,k,s \in \NN$ satisfying $1 \le k \le m - 1$ and $\gcd(m,s) = 1$, let $\Gamma(m,k,s) = \<m, m + s, \ldots, m + ks\>$, and let $S_{m,k}^s = S_{\Gamma(m,k,s)}^s$.  
\end{defn}




We begin by giving a closed form for the first infinite column of an arithmetical Leamer monoid.  

\begin{prop}\label{p:arithxf}
Fix an arithmetical Leamer monoid $S_{m,k}^s$.  The first infinite column $x_f$ of $S_{m,k}^s$ is given by 
$$x_f = \mathcal F(\Gamma) - \mathcal F(\<m,s\>) = m(\lfloor \textstyle\frac{m-2}{k} \rfloor + 1).$$
\end{prop}

\begin{proof}
The second equality follows from \cite[Theorem~3.3.2]{diophantine}.  For $x < \mathcal F(\Gamma) - \mathcal F(\<m,s\>)$, we can write $\mathcal F(\Gamma) = x + am + bs$ for some $a, b \ge 0$.  This means the column at $x + am$ is finite, so the column at $x$ is finite.  Now let $x = \mathcal F(\Gamma) - \mathcal F(\<m,s\>)$.  By \cite[Theorem~3.1]{omidali}, $x + rs \in \Gamma$ for $0 \le r \le k(\lfloor \frac{m-2}{k} \rfloor + 1)$.  
Notice that 
$$s(k\lfloor \textstyle\frac{m-2}{k} \rfloor + k) \ge s(m - 2 - k + 1 + k) = s(m-1) > \mathcal F(\<m,s\>)$$
so $x + sk(\lfloor \frac{m-2}{k} \rfloor + 1) > \mathcal F(\Gamma)$.  This completes the proof.  
\end{proof}

We now give a complete characterization of the elements and atoms of arithmetical Leamer monoids.  

\begin{thm}\label{t:arith}
Fix an arithmetical Leamer monoid $S_{m,k}^s$, fix $\alpha, i \in \mathbb{N}$ with $0 \le i \le m - 1$, and further let $x = \alpha m + is$.  Then we have the following.  

\begin{enumerate}
\item[(a)] 
$S_{m,k}^s$ has a finite column at $x$ if and only if $k\alpha \le m-2$ and $0 \le i \le k\alpha - 1$.  In this case, the column at $x$ has height $k\alpha - i$.  

\item[(b)] 
$S_{m,k}^s$ has an infinite column at $x$ if and only if $k\alpha \ge m - 1$.  

\item[(c)] 
If the column at $x$ is finite, then it has nontrivial atoms if and only if $\alpha = 1$ and $k \ge 2$, in which case it consists entirely of atoms.

\item[(d)] 
If the column at $x$ is infinite, then it has nontrivial irreducibles if and only if $\alpha = \lfloor \frac{m - 2}{k} \rfloor + 1$.  In this case, the first nontrivial irreducible in $x$ has height $\max\{2, k\alpha - i + 1\}$.  

\end{enumerate}
\end{thm}

\begin{proof}
By \cite[Theorem~3.1]{omidali}, each $x \in \Gamma$ can be written as $x = \alpha m + is$ for some $i, \alpha \in \NN$ with $0 \le i \le m - 1$.  
If $k\alpha > m-2$, then $\alpha > \lfloor \frac{m-2}{k} \rfloor + 1$, so 
$$x - x_f = m(\alpha - \lfloor \textstyle\frac{m-2}{k} \rfloor + 1) + is \in \Gamma.$$
It then follows by Proposition~\ref{p:arithxf} that $S_{m,k}^s$ has an infinite column at $x$.  If $k\alpha \le m-2$, then by \cite[Theorem~3.1]{omidali}, we must have $0 \le i \le k\alpha$.  In particular, we have $\alpha m + (k\alpha + 1)s \notin \Gamma$ since this is a unique factorization in $\<m,s\>$.  This means $S_\Gamma^s$ has a finite column at $\alpha m + is$ of height $k\alpha - i$ for each $0 \le i \le k\alpha - 1$.  By \cite[Theorem~3.1]{omidali}, we have considered every element of $\Gamma$, and thus have accounted for every column of $S_{m,k}^s$.  This proves~(a) and~(b).  

Now, suppose the column at $x = \alpha m + is$ is finite.  If $\alpha = 1$, then $x$ is a minimal generator for $\Gamma$, so every element in the column at $x$ is irreducible.  If $\alpha > 1$ and the column at $x$ does not have height 1, then for $r = \min\{i, k-1\}$, we can write 
$$(x,k\alpha - i) = (m + rs, \alpha - r) + ((\alpha - 1)m + (i-r)s, \alpha(k - 1) - (i - r)),$$
so the top element in the column at $x$ is reducible, and by Lemma~\ref{l:atoms}(d), so are the nontrivial elements below it.

Finally, suppose the column at $x = \alpha m + is$ is infinite.  If $\alpha > \lfloor \frac{m - 2}{k} \rfloor + 1$, then there is an infinite column at $x - m$, so the column at $x$ is purely reducible.  Now suppose $\alpha = \lfloor \frac{m - 2}{k} \rfloor + 1$.  The column at $x - m$ is either empty or finite, so the column at $x$ must contain nontrivial irreducibles.  Let $C_x$ denote the height of the lowest nontrivial irreducible element in the column at $x$.  For any expression $x = \sum_{j=1}^r (\alpha_j m + i_j s)$ of $x$ in terms of finite columns $\alpha_j m + i_j s$, we have $C_x - 1 \ge \sum_{j=1}^r (k\alpha_j - i_j) = k\alpha - i$, which gives $C_x \ge k\alpha - i + 1$ with equality as long as $k\alpha - i + 1 \ge 2$, as desired.  
\end{proof}

Theorem~\ref{t:arith} shows that, in particular, each $S_{m,k}^s$ has an irreducible of height 2.  This, together with Proposition~\ref{p:HW}, yields the following.  

\begin{cor}
Let $R = \mathbb K[\Gamma(m,k,s)]$, and $I = \<t^a, t^{a+s}\>$ be a monomial ideal with $a \in \Gamma(m,k,s)$.  Then the ideal $I$ satisfies the Huneke-Wiegand conjecture.  
\end{cor}

We now give a closed form for the Delta sets for this class of Leamer monoids.  

\begin{thm}\label{t:turtle}
Fix an arithmetical Leamer monoid $S_{m,k}^s$.  Then 
$$\Delta(S_{m,k}^s) = \{1, \ldots, \lfloor \textstyle\frac{m-2}{k} \rfloor + 1\}.$$
\end{thm}

\begin{proof}

Suppose $k > 1$.  Let $x_T = m(\lfloor \frac{m-2}{k} \rfloor + s + 2)$ and fix $\beta \le \lfloor \frac{m-2}{k} \rfloor + 1$.  We claim $\mathcal L(x_T,k\beta + 2) \cap \{2, \ldots, \beta + 2\} = \{2, \beta + 2\}$.  By Theorem~\ref{t:arith}, we have factorizations 
\begin{eqnarray*}
(x_T,k\beta + 2) 
&=& (m + s,1) + (m(\lfloor \textstyle\frac{m-2}{k} \rfloor + 1) + (m-1)s, k\beta + 1) \\
&=& \beta (m,k) + (m,1) + (m(\lfloor \textstyle\frac{m-2}{k} \rfloor + 1 - \beta + s), 1)
\end{eqnarray*}
so $2, \beta + 2 \in \mathcal L(x_T,k\beta + 2)$.  Now, in order for a nontrivial irreducible in the column at $x = \alpha m + is$ to appear in a factorization of $(x_T,k\beta + 2)$, we must have a column at $$x_T - x = m(\lfloor \textstyle\frac{m-2}{k} \rfloor + 2 - \alpha) + (m - i)s.$$
If the column at $x$ is infinite, then by Theorem~\ref{t:arith}, $\alpha = \lfloor \frac{m-2}{k} \rfloor + 1$, so $x_T - x = m + (m - i)s$.  This means any factorization containing a nontrivial irreducible from $x$ must have length 2.  Thus, any nontrivial irreducible in a factorization of length greater than 2 must lie in a finite column.  

We claim that $(x_T, k\beta + 2)$ has no factorization of length at most $\beta + 1$ consisting entirely of irreducibles in columns of the form $m + is$ for some $i \le k$.  Fix a sum $\sum_{j = 1}^r (m + i_js, n_j)$ of irreducibles with $\sum_{j = 1}^r n_j = k\beta + 2$.  By Theorem~\ref{t:arith}, $n_j \le k - i_j$ for each $j$, so this sum can have height at most $rk - \sum_{j = 1}^r i_j$.  This means
$$\sum_{j = 1}^r m + i_js = rm + s\sum_{j = 1}^r i_j \le rm + s\sum_{j = 1}^r (k - n_j) = rm + rsk - s(k\beta + 2)$$
so if $r \le \beta + 1$, this can be at most
$$m(\beta + 1) + sk(\beta + 1) - s(k\beta - 2) = m(\beta + 1) + s(k - 2) < m(\beta + 1 + s) \le x_T.$$
This means there can be no factorization of $(x_T,k\beta + 2)$ of this form.  

By the above, every factorization of $(x_T,k\beta + 2)$ of length at most $\beta + 1$ must contain a trivial irreducible.  But this means the second coordinates of the (at most $\beta$) remaining factors must sum to $k\beta + 1$, which is impossible since each such coordinate can be at most $k$.  Thus, there can be no factorization of $(x_T,k\beta + 2)$ of length strictly between $2$ and $\beta + 2$, meaning $\beta~\in~\Delta(x_T, k\beta + 2)$.  

Now, for $k = 1$, $x_T = m(m + s + 1)$ and $\beta \le m - 1$, a similar argument shows that $\mathcal L(x_T, \beta + 2) = \{2, \beta + 2\}$.  Thus, for any $k \le m-1$, $\Delta(S_{m,k}^s) \subset \left\{1, \ldots, \left\lfloor \frac{m-2}{k} \right\rfloor + 1\right\}$.  By \cite[Lemma~4.1]{schaeffer} and Theorem~\ref{t:arith}, we have $\max\Delta(S_{m,k}^s) \le \lfloor \textstyle\frac{m-2}{k} \rfloor + 1$, and this completes the proof.
\end{proof}

We have the following immediate corollary.  For a full treatment on the catenary degree, see~\cite{bible} or~\cite{omidali}.  

\begin{cor}\label{c:catenary}
The arithmetical Leamer monoid $S_{m,k}^s$ has catenary degree $\lfloor \textstyle\frac{m-2}{k} \rfloor + 3$.  
\end{cor}

\begin{proof}
This follows from Theorem~\ref{t:turtle}, \cite[Lemma~4.1]{schaeffer} and \cite[Lemma~1.6.2]{bible}.  
\end{proof}

\section{Future Work}\label{s:future}

This paper develops many foundational theorems and tools related to Leamer monoids, which have a very rich factorization theory.  Thus, many open problems and questions about Leamer monoids remain. 

\begin{prob}\label{p:algorithm}
Find an algorithm to compute $\Delta(S_\Gamma^s)$ for any Leamer monoid.
\end{prob}

\begin{question}\label{q:deltainterval}
For every Leamer monoid $S_\Gamma^s$, does there exists an $M$ such that $\Delta(S_\Gamma^s) = \{1, 2, \ldots, M\}$?  
\end{question}

\begin{prob}\label{p:catenary}
Find the catenary degree $c(S_\Gamma^s)$ for any Leamer monoid $S_\Gamma^s$.  
\end{prob}

A new invariant measuring how far an element is from being prime, called $\omega$-primality, has been studied in several different settings, including numerical monoids (see~\cite{prime, bounding, tame, quasilinear}).  A natural problem is to study $\omega$-primality in Leamer monoids.

\begin{prob}\label{p:omega}
Study the $\omega$-function $\omega(x,n)$ for elements $(x,n) \in S_\Gamma^s$ in any Leamer monoid $S_\Gamma^s$.  
\end{prob}

Understanding Leamer monoids associated to numerical monoids with special generators is also of interest.  See~\cite[Lemma~1]{interval} and~\cite[Theorem~3.1]{omidali} for relevant membership criteria.  

\begin{prob}\label{p:generatedgenarith}
A \emph{generalized arithmetic sequence} is a sequence of the form 
$$\{a, ah+d, ah + 2d, \ldots, ah+kd\}$$
where $a,k,h,d \in \mathbb N$ with $\gcd(a,d) = 1$ and $a \geq 2$.  Study Leamer monoids $S_\Gamma^s$ where $\Gamma$ is generated by a generalized arithmetic sequence.  
\end{prob}

\begin{prob}\label{p:generatedinterval}
Study Leamer monoids $S_\Gamma^s$ where $\Gamma$ is generated by an interval of natural numbers.  
\end{prob}

\section{Acknowledgements}\label{s:acknowledgements}
\raggedbottom

Much of this work was completed during the Pacific Undergraduate Research Experience in Mathematics (PURE Math), which was funded by National Science Foundation grants DMS-1035147 and DMS-1045082 and a supplementary grant from the National Security Agency.  The authors would like to thank Scott Chapman, Pedro Garc\'ia-S\'anchez, and Micah Leamer for their numerous helpful conversations, as well as the anonymous referee for their very helpful comments.


\end{document}